\documentclass[]{interact}

\usepackage{enumerate}

\usepackage[numbers,sort&compress]{natbib}
\bibpunct[, ]{[}{]}{,}{n}{,}{,}
\renewcommand\bibfont{\fontsize{10}{12}\selectfont}
\makeatletter
\def\NAT@def@citea{\def\@citea{\NAT@separator}}
\makeatother

\theoremstyle{plain}
\newtheorem{theorem}{Theorem}[section]

\theoremstyle{definition}
\newtheorem{definition}[theorem]{Definition}
\newtheorem{example}[theorem]{Example}

\theoremstyle{remark}
\newtheorem*{remark}{Remark}
\newtheorem*{remarks}{Remarks}

\newcommand{\CC}{{\mathbb C}}

\newcommand{\DD}{{\mathbb D}}
\DeclareMathOperator{\inter}{int}
\DeclareMathOperator{\diag}{diag}
\renewcommand{\Re}{\operatorname{Re}}

\begin{document}

\title{Spectral sets, extremal functions and exceptional matrices}

\author{
\name{
Thomas Ransford*
\thanks{*Communicating author, email: ransford@mat.ulaval.ca} 
and Nathan Walsh}
\affil{D\'epartement de math\'ematiques et statistique, Universit\'e Laval, Qu\'ebec (QC) G1V 0A6, Canada}
}

\maketitle

\begin{abstract}
Let $A$ be a square matrix and let $\Omega$ be an open set in the plane containing the spectrum of $A$.
We consider the problem of maximizing the operator norm  $\|f(A)\|$ amongst all holomorphic functions
$f$ from $\Omega$ into the closed unit disk. If $f_0$ is extremal for this problem and if $\|f_0(A)\|>1$, 
then it turns out that the matrix $f_0(A)$ has special properties, 
among them the fact that its principal left and right singular vectors 
are mutually orthogonal. We study this class of exceptional matrices $f_0(A)$.
In particular, we are interested in the extent 
to which they are characterized by the aforementioned orthogonality property.
\end{abstract}


\begin{amscode}
15A60
\end{amscode}

\section{Introduction}

\subsection{Spectral sets}

Let $A$ be a  complex $N\times N$  matrix.
We write $\|A\|$ for the operator norm of~$A$, considered as an operator on $\CC^N$ with the Euclidean norm.
Also we denote by $\sigma(A)$  the spectrum of $A$, namely its set of eigenvalues.

A closed  subset $X$ of $\CC$ is said to be a \emph{$K$-spectral set} for $A$ if $\sigma(A)\subset X$ and if,
for every rational function $f$ with poles outside $X$, we have
\begin{equation}\label{E:Kspectral}
\|f(A)\|\le K\sup_X|f|.
\end{equation}
By considering $f\equiv 1$, we see that necessarily $K\ge1$.
If \eqref{E:Kspectral} holds  with $K=1$, then $X$ is called simply a \emph{spectral set}.

The notion of spectral set was introduced by von Neumann \cite{vN51}.
He showed that, if $\|A\|\le1$, then the closed unit disk $\overline{\DD}$ is a spectral set for $A$. 
Using standard approximation arguments, one can reformulate this result in various equivalent ways.
For example, if $\|A\|\le1$ and $\sigma(A)\subset\DD$, then, for every bounded holomorphic
$f$ on $\DD$, we have
\begin{equation}\label{E:vN}
\|f(A)\|\le \sup_{\DD}|f|.
\end{equation}
Here $f(A)$ is defined using the usual holomorphic functional calculus.
There are now many proofs of von Neumann's result; two of our favourites are in \cite{Ne61,DD99}.

Another example of a spectral set is furnished by the spectral theorem, 
namely the fact that a normal matrix is unitarily equivalent to a diagonal matrix. 
From this it follows easily that, if $A$ is a normal matrix, 
then $\sigma(A)$ is itself a spectral set for $A$.

A third and more recent example arises in connection with the notion of numerical range. 
Recall that the \emph{numerical range} $W(A)$ of $A$ 
is the set of values of $\langle Ax,x\rangle$ 
as $x$ runs through the  vectors of $\CC^N$ of unit Euclidean norm. 
It is well known that $W(A)$ is a compact convex set containing $\sigma(A)$. 
Crouzeix and Palencia \cite{CP17} showed that $W(A)$  is $(1+\sqrt{2})$-spectral set for $A$, 
improving earlier results of Delyon and Delyon \cite{DD99} and Crouzeix \cite{Cr07}.
Crouzeix \cite{Cr04} proved that $W(A)$ is  even a $2$-spectral set for $A$ if $N=2$,
and conjectured that this is in fact true for all $N\ge3$.
This remains an interesting open problem.

For more information about spectral sets, 
we refer to the survey article of Badea and Beckermann \cite[Chapter 37]{BB14}.

\subsection{Extremal functions}\label{S:extremal}

Let $A$ be an $N\times N$ matrix, and let
$\Omega$ be an open subset of $\CC$ containing $\sigma(A)$.
A standard normal-families argument shows that there exists a
holomorphic function $f_0:\Omega\to\overline{\DD}$ such that $\|f(A)\|\le\|f_0(A)\|$
for all other holomorphic functions $f:\Omega\to\overline{\DD}$.
We call such a function $f_0$ \emph{extremal} for the pair $(A,\Omega)$.

It follows from this observation that, if $\sigma(A)$ is contained in the interior of
a closed subset $X$ of $\CC$ and if $f_0$ is extremal for 
the pair $(A,\inter(X))$,
then $X$ is automatically a $K$-spectral set for $A$, 
where $K:=\|f_0(A)\|$. There is thus a certain interest in identifying $f_0$ 
and estimating $\|f_0(A)\|$.

It was shown in \cite{Cr04} and again in \cite{CGH14} that, 
if $\Omega$ is simply connected and $\phi$ is a conformal map of $\Omega$ onto~$\DD$,
then every extremal function $f_0$ for $(A,\Omega)$ has the form $f_0=b\circ\phi$, 
where  $b$ is a Blaschke product of degree at most $N-1$. 
In some cases the Blaschke product $b$ is unique up to multiplication by unimodular constants, 
in other cases not.

For further information on extremal functions we refer to \cite{BGGRSW20}.

\subsection{Exceptional matrices}

In this article, the focus of our attention is not so much on the extremal functions $f_0$ themselves, 
but rather on the matrices $f_0(A)$. Can we characterize these matrices?

Once again, let $A$ be an $N\times N$ matrix, 
let $\Omega$ be a connected open set containing $\sigma(A)$, 
let $f_0$ be an extremal function for the pair $(A,\Omega)$, and set $E:=f_0(A)$. 
If $|f_0|$ attains the value $1$ at some point of $\Omega$, 
then by the maximum principle $f_0$ is constant, 
and $E$ is just a multiple of the identity matrix. 
Otherwise we have $f_0(\Omega)\subset\DD$, 
and in that case $E$ has the following properties:
\begin{enumerate}
\item[(i)] $\sigma(E)\subset\DD$;
\item[(ii)] $\|h(E)\|\le\|E\|$ for all holomorphic functions $h:\DD\to\overline{\DD}$.
\end{enumerate}
Indeed, property~(i) is an immediate consequence of the spectral mapping theorem, 
and property~(ii) follows directly from the definition of extremal function. 

Conversely, every $N\times N$ matrix $E$ satisfying (i) and (ii) is of the form $f_0(A)$ 
for some pair $(A,\Omega)$ and
some extremal function $f_0$ for $(A,\Omega)$. 
Indeed, we can just take $A:=E$ and $\Omega:=\DD$ and $f_0(z):=z$.

This would appear to answer the question posed at the beginning of the subsection.
However, it still leaves us with the problem of determining 
which matrices $E$ satisfy property~(ii). 
Notice that this property, applied  with $h\equiv1$, implies that $\|E\|\ge1$,
so the problem divides into two sub-cases, namely $\|E\|=1$ and $\|E\|>1$.

The first sub-case is easy. If $\sigma(E)\subset\DD$ and $\|E\|=1$, 
then (ii) automatically holds by von Neumann's inequality \eqref{E:vN}.

The second sub-case, namely $\|E\|>1$, is more interesting, 
especially in view of the ultimate goal of estimating $\|E\|=\|f_0(A)\|$. 
It turns out that matrices satisfying (i), (ii) and $\|E\|>1$ are somewhat special.
We are led to formulate the following definition.

\begin{definition}\label{D:exceptional}
An $N\times N$ matrix $E$ is \emph{exceptional} if it has the following properties:
\begin{enumerate}
\item[(i)] $\sigma(E)\subset\DD$,
\item[(ii)] $\|h(E)\|\le\|E\|$ for all holomorphic functions $h:\DD\to\overline{\DD}$,
\item[(iii)] $\|E\|>1$.
\end{enumerate}
\end{definition}

The purpose of this article is to study exceptional matrices:
to exhibit examples of such matrices, to establish their properties, and to try to characterize them.
We shall end up with a complete characterization when $N=2$ and a partial characterization when $N\ge3$.

\section{Examples of exceptional matrices}

When $N=1$, the two conditions (i) and (iii) in Definition~\ref{D:exceptional} are mutually exclusive.
Thus there are no $1\times 1$ exceptional matrices. Henceforth we assume that $N\ge2$.

Our first result provides some basic examples of exceptional matrices.

\begin{theorem}\label{T:NxNsuff}
Let $N\ge2$ and let $E$ be a matrix of the form
\begin{equation}\label{E:NxNsuff}
E:=
\left(
\begin{array}{c|c}
0 & \raisebox{-15pt}{{\Large\mbox{{$A$}}}} \\[-3ex]
\vdots & \\[-0.ex]
0 &\\ \hline
a& 0 \dots 0 
\end{array}
\right),
\end{equation}
where $a >1$ and $A$ is an $(N-1)\times(N-1)$ matrix with $\|A\|<1/a$.
Then $E$ is exceptional.
\end{theorem}

\begin{proof}
First of all, we note that $\|E\|=\max\{a,\|A\|\}=a$. In particular, $\|E\|>1$.

Next, we remark that $E$ may be factorized as $E=VJV^{-1}$, where $V=\diag(1,1,\dots,1,a)$,
where
\[
J=\left(
\begin{array}{c|c}
0 & \raisebox{-15pt}{{\Large\mbox{{$AD$}}}} \\[-3ex]
\vdots & \\[-0.ex]
0 &\\ \hline
1& 0 \dots 0 
\end{array}
\right),
\]
and where $D$ is the $(N-1)\times (N-1)$ diagonal matrix $D=\diag(1,1,\dots,a)$. 
Notice that $\|AD\|\le\|A\|\|D\|=\|A\|a<1$.
It follows that $\|J\|=\max\{1,\|AD\|\}=1$.  Moreover, 
writing $\{e_1,\dots,e_N\}$ for the standard unit vector basis of $\CC^N$, we see that
$\|Jx\|<\|x\|$ for all vectors $x$ that are not scalar multiples of $e_1$,
and $Je_1=e_N$, so $\|J^2x\|<\|x\|$ for all non-zero vectors $x\in\CC^N$. 
In particular, it follows that $\sigma(J)\subset\DD$, and, since $E$ is similar to $J$, 
we therefore also have $\sigma(E)\subset\DD$.

Finally, given a holomorphic function $h:\DD\to\overline{\DD}$, 
we have $\|h(J)\|\le 1$ by von Neumann's inequality \eqref{E:vN}, and hence
\[
\|h(E)\|=\|h(VJV^{-1})\|=\|Vh(J)V^{-1}\|\le\|V\|\|h(J)\|\|V^{-1}\|\le a.1.1=\|E\|.
\]
This completes the verification that $E$ is exceptional.
\end{proof}

Here are two examples
concerning the sharpness of the condition $\|A\|<1/a$ in Theorem~\ref{T:NxNsuff}.

\begin{example}\label{Ex:ex1}
 Let $E$ be a matrix of the form \eqref{E:NxNsuff}, 
where $a>1$ and $A$ consists entirely of zeros except for an entry $\alpha$ in the top right-hand corner.
Then $E^2$ has $a\alpha$ as an eigenvalue, so, if $E$ is to be exceptional, then the condition that $\sigma(E)\subset\DD$
forces $|a\alpha|<1$, in other words $\|A\|<1/a$. This example shows that the constant $1/a$ in the condition $\|A\|<1/a$ is sharp.
\end{example}

\begin{example}\label{Ex:ex2}
 Again let $E$ be a matrix of the form \eqref{E:NxNsuff}, where $a>1$ and now $A$ has just zeros in the last column.
In the proof of Theorem~\ref{T:NxNsuff} above, we then have $\|AD\|\le \|A\|$, and so, if we assume merely that $\|A\|<1$ (rather than $1/a$), then the rest of the proof goes through to show that $E$ is exceptional. 
This example demonstrates that the sufficient condition $\|A\|<1/a$ is not necessary for $E$ to be exceptional.
\end{example}

The next result lists a number of different ways of constructing new exceptional matrices from old ones.

\begin{theorem}\label{T:newfromold}
\begin{enumerate}
\item[\rm (a)] If $E$ is exceptional, then its complex conjugate $\overline{E}$, its transpose $E^t$ and its adjoint $E^*$ are all exceptional.
\item[\rm (b)]\label{T:unitary} If $E$ is exceptional and if $U$ is unitary, then $U^*EU$ is exceptional.
\item[\rm(c)] If $E$ is  exceptional and if $h_0:\DD\to\DD$ is a holomorphic map such that
$\|h_0(E)\|=\|E\|$, then $h_0(E)$ is exceptional. In particular, $\alpha E$ is exceptional  whenever $\alpha\in\CC$ with  $|\alpha|=1$.
\item[\rm (d)] If $E,F$ are exceptional (perhaps of different dimensions), then the block matrix
$E\oplus F$ is exceptional.
\item[\rm(e)] If $E$ is exceptional and $F$ is a square matrix satisfying $\sigma(F)\subset\DD$ and $\|F\|\le1$,
 then the block matrix
$E\oplus F$ is exceptional.
\end{enumerate}
\end{theorem}

\begin{proof}
The proofs are all routine verifications. We omit the details.
\end{proof}

Finally we record a result that reduces the  amount of work required to
verify that a given matrix is exceptional, and opens 
the door to numerical testing in low dimensions.

\begin{theorem}
Let $E$ be an $N\times N$ matrix such that $\sigma(E)\subset\DD$ and $\|E\|>1$.
Then $E$ is exceptional if and only if $\|b(E)\|\le\|E\|$
for every Blaschke product $b$ of degree at most $N-1$.
\end{theorem}

\begin{proof}
The `only if' is obvious. The `if' part follows from the result about extremal functions
cited in \S\ref{S:extremal}.
\end{proof}

\section{Properties of exceptional matrices}

As their name suggests, exceptional matrices enjoy a number of special properties.
We now investigate these. Perhaps the most significant one is the following
orthogonality theorem, which is a translation into our language of a result in \cite{CGH14}.

\begin{theorem}\label{T:orthog}
Let $E$ be an $N\times N$ exceptional matrix, and let $x_0$ be a unit vector in $\CC^N$ 
such that $\|Ex_0\|=\|E\|$. Then
\begin{equation}\label{E:orthog}
\langle Ex_0,x_0\rangle=0.
\end{equation}
\end{theorem}

\begin{remarks}
(1) This theorem can be formulated as saying that the principal left and right singular vectors of an exceptional matrix
are always orthogonal to each other.

(2) Apparently unaware of \cite{CGH14}, the authors of  \cite{CGL18} rediscovered this result, and used it to
give a very simple proof of the theorem of Okubo and Ando \cite{OA75} to the effect that, if $W(A)\subset\overline{\DD}$, then
$\overline{\DD}$ is a $2$-spectral set for $A$.
\end{remarks}

Since Theorem~\ref{T:orthog} plays a crucial role in what follows, 
we include a short proof for the convenience of the reader.

\begin{proof}[Proof of Theorem~\ref{T:orthog}]
For $w\in\DD$, let $h_w(z):=(z-w)/(1-\overline{w}z)$. 
Then $h_w$ maps $\DD$ into $\DD$ and, for each $z\in\DD$,
\[
h_w(z)=z-w+\overline{w}z^2+O(|w|^2)
\quad(w\to0).
\]
As $E$ is exceptional, we have $\|h_w(E)\|\le\|E\|$, and so, since $x_0$ is a unit vector,
\[
|\langle h_w(E)x_0,Ex_0\rangle|\le \|h_w(E)x_0\|\|Ex_0\|\le\|h_w(E)\|\|E\|\le\|E\|^2.
\]
Hence
\begin{align*}
\|E\|^2&\ge
\Re \langle h_w(E)x_0,Ex_0\rangle\\
&=\Re \langle (E-wI+\overline{w}E^2+O(|w|^2))x_0,Ex_0\rangle\\
&=\|Ex_0\|^2-\Re\bigl(w\langle x_0,Ex_0\rangle\bigr)+ \Re\bigl(\overline{w}\langle E^2x_0,Ex_0\rangle\bigr)+O(|w|^2)\\
&=\|Ex_0\|^2-\Re\bigl(\overline{w}\langle Ex_0, x_0\rangle\bigr)+ \Re\bigl(\overline{w}\langle Ex_0,E^*Ex_0\rangle\bigr)+O(|w|^2)\\
&=\|Ex_0\|^2+\Re\bigl(\overline{w}\langle Ex_0,(E^*E-I)x_0\rangle\bigr)+O(|w|^2).
\end{align*}
Since $x_0$ is a unit vector where $E$ attains its norm, we have $\|Ex_0\|=\|E\|$ and $E^*Ex_0=\|E\|^2x_0$.
Hence
\[
\|E\|^2\ge \|E\|^2+(\|E\|^2-1)\Re\bigl(\overline{w}\langle Ex_0,x_0\rangle\bigr)+O(|w|^2).
\]
Cancelling off the $\|E\|^2$ terms and dividing through by $(\|E\|^2-1)$,
which is strictly positive, we obtain
\[
\Re\bigl(\overline{w}\langle Ex_0,x_0\rangle\bigr)\le O(|w|^2).
\]
Letting $w\to0$ and noting that the argument of $w$ is arbitrary, we get \eqref{E:orthog}.
\end{proof}

The next theorem is also a translation into our language of a result in \cite{CGH14}.
Once again, we include a short proof for the convenience of the reader 
(a bit different from the one given in \cite{CGH14}).

\begin{theorem}\label{T:positive}
Let $E$ be an $N\times N$ exceptional matrix, 
and let $x_0$ be a unit vector in $\CC^N$ such that $\|Ex_0\|=\|E\|$. Then,
for every holomorphic function $h:\DD\to\DD$,
\begin{equation}\label{E:positive}
|\langle h(E)x_0,x_0\rangle|\le 1.
\end{equation}
\end{theorem}

\begin{proof}
For $t>0$, let 
\[
g_t(z):=z\exp\Bigl(-t(1-h(z))\Bigr).
\]
Note that $1-h(z)$ has positive real part for all $z\in\DD$, so $g_t$ maps $\DD$ into $\DD$.
Also, for each $z\in\DD$,
\[
g_t=z-tz(1-h(z))+O(t^2) \quad(t\to0^+).
\]
As $E$ is exceptional, we have $\|g_t(E)\|\le\|E\|$, and so, since $x_0$ is a unit vector,
\[
|\langle g_t(E)x_0,Ex_0\rangle|\le \|g_t(E)x_0\|\|Ex_0\|\le\|g_t(E)\|\|E\|\le\|E\|^2.
\]
Hence
\begin{align*}
\|E\|^2&\ge
\Re \langle g_t(E)x_0,Ex_0\rangle\\
&=\Re \langle (E-tE(I-h(E))+O(t^2))x_0,Ex_0\rangle\\
&=\|Ex_0\|^2-t\Re\bigl\langle (E(I-h(E)) x_0,Ex_0\rangle\bigr)+O(t^2)\\
&=\|Ex_0\|^2-t\Re\bigl\langle (I-h(E)) x_0,E^*Ex_0\rangle\bigr)+O(t^2).
\end{align*}
Since $x_0$ is a unit vector on  which $E$ attains its norm, 
we have $\|Ex_0\|=\|E\|$ and $E^*Ex_0=\|E\|^2x_0$.
Hence
\[
\|E\|^2\ge \|E\|^2-t\|E\|^2\Re\langle (I-h(E))x_0,x_0\rangle+O(t^2).
\]
Cancelling off the $\|E\|^2$ terms and dividing through by $t$,
and letting $t\to0^+$, we get 
\[
\Re\langle h(E)x_0,x_0\rangle\le 1.
\]
Repeating the argument with $h$ replaced by $e^{i\theta}h$, and letting $\theta$ vary, we finally obtain
\[
|\langle h(E)x_0,x_0\rangle|\le 1.\qedhere
\]
\end{proof}

We conclude this section by remarking that the ideas behind Theorems~\ref{T:orthog} and \ref{T:positive} are developed further in \cite{BGGRSW20}. In particular, it is shown in \cite[Theorem~4.5]{BGGRSW20} that, under the hypotheses
of Theorem~\ref{T:positive}, there exists a unique probability measure $\mu$ on $\partial \DD$ such that,
for all continuous functions $h$ on $\overline{\DD}$ that are holomorphic on $\DD$, we have
\[
\langle h(E)x_0,x_0\rangle =\int_{\partial\DD}h\,d\mu.
\]

\section{Characterization of exceptional matrices}

In this section, we amalgamate the ideas from the two preceding sections to try to characterize
exceptional matrices.
The following result furnishes a necessary condition and a sufficient condition that resemble each other.
The necessary condition is inspired by a result in \cite{CGH14}.

\begin{theorem}\label{T:NxNnec}
Let $N\ge 2$ and let $E$ be a $N\times N$  matrix. 
\begin{enumerate}
\item[\rm(a)]
If $E$ is exceptional, then
$E$ is unitarily equivalent to a matrix of the form
\begin{equation}\label{E:NxNnec}
\left(
\begin{array}{c|c}
0 & \raisebox{-15pt}{{\Large\mbox{{$A$}}}} \\[-3ex]
\vdots & \\[-0.ex]
0 &\\ \hline
a& 0 \dots 0 
\end{array}
\right),
\end{equation}
where $a=\|E\|$ and $A$ is an $(N-1)\times(N-1)$ matrix with $\|A\|\le a$.
\item[\rm(b)]
If $E$ is unitarily equivalent to a matrix of the form \eqref{E:NxNnec},
where $a>1$ and  $A$ is an $(N-1)\times(N-1)$ matrix with $\|A\|<1/a$,
then $E$ is exceptional and $\|E\|=a$.
\end{enumerate}
\end{theorem}

\begin{proof}
(a) Let $y_1$ be a unit vector in $\CC^N$ on which $E$ attains its norm, and set 
$y_N:=Ey_1/\|E\|$. Clearly $y_N$ is a unit vector, and, by Theorem~\ref{T:orthog},
the vectors $y_1$ and $y_N$ are orthogonal. We may therefore extend them to an orthonormal 
basis $\{y_1,y_2,\dots,y_N\}$  of $\CC^N$.

Now, for $1\le j<N$, we have 
\begin{equation}\label{E:orthog1}
\langle Ey_1,y_j\rangle=\|E\|\langle y_N,y_j\rangle=0.
\end{equation}
Also, since $E$ attains its norm on $y_1$, we have  $E^*Ey_1=\|E\|^2y_1$,
and consequently, for $1<j\le N$, we have
\begin{equation}\label{E:orthog2}
\langle y_N,Ey_j\rangle=\langle Ey_1,Ey_j\rangle/\|E\|=\langle E^*Ey_1,y_j\rangle/\|E\|=\|E\|\langle y_1,y_j\rangle=0.
\end{equation}
The orthogonality relations \eqref{E:orthog1} and \eqref{E:orthog2} imply that 
$E$ has a matrix of the form \eqref{E:NxNnec} with respect to the orthonormal basis $\{y_1,\dots,y_N\}$,
in other words, that $E$ is unitarily equivalent to a matrix of that  form.

(b) This is proved simply by combining Theorem~\ref{T:NxNsuff} and Theorem~\ref{T:newfromold}\,(b).
\end{proof}

Theorem~\ref{T:NxNnec} is only a partial characterization of exceptional matrices,
because of the gap between the necessary condition $\|A\|\le a$ and the sufficient condition $\|A\|\le 1/a$.
Examples~\ref{Ex:ex1} and \ref{Ex:ex2}  hint at some  obstacles to closing this gap. 
However, there is one case where we can close the gap completely, namely when $N=2$.
The following result exhibits several characterizations of $2\times 2$ exceptional matrices.

\begin{theorem}\label{T:2x2}
Let $E$ be a  $2\times 2$ matrix such that $\sigma(E)\subset\DD$ and $\|E\|>1$. 
The following statements are equivalent:
\begin{enumerate}
\item[\rm (a)] $E$ is exceptional;
\item[\rm (b)] for every unit vector $x_0\in\CC^2$ such that $\|Ex_0\|=\|E\|$, we have $\langle Ex_0,x_0\rangle=0$;
\item[\rm(c)] $E$ is unitarily equivalent to a matrix with zeros on the diagonal;
\item[\rm (d)] $E$ has trace zero.
\end{enumerate}
\end{theorem}

\begin{proof}
(a)$\Rightarrow$(b). 
This is just a special case of Theorem~\ref{T:orthog}.

\medskip
(b)$\Rightarrow$(c). 
Repeat the proof of Theorem~\ref{T:NxNnec}.

\medskip
(c)$\Rightarrow$(a). 
Since unitary equivalence preserves the property of being exceptional,
it is enough to show that every $2\times 2$ matrix $E$ with zero entries on the diagonal is exceptional.
Let $a$ and $b$ be the off-diagonal entries of $E$. Without loss of generality $|a|\ge|b|$.
Conjugating $E$ by a suitable unitary matrix, we may also suppose that $a\ge0$.
Since $\|E\|>1$, we have $a>1$. Also, since $\sigma(E)\subset\DD$, we have $|\det(E)|<1$, whence $|b|<1/a$.
Thus either $E$ or its transpose has the form \eqref{E:NxNsuff}, and so by Theorem~\ref{T:NxNsuff} $E$ is exceptional.

\medskip
(c)$\Longleftrightarrow$(d).
The forward direction is obvious. The reverse direction is an
exercise using the Schur factorization (every square matrix is unitarily equivalent to a triangular matrix).
\end{proof}
 
\begin{remark}
The equivalence between conditions (b) and (c) in Theorem~\ref{T:2x2} was remarked by the authors 
of \cite{CGL18}, who used it to give a simplified proof of the Crouzeix conjecture for $2\times 2$ matrices.
(Their remark does not appear in the published paper \cite{CGL18}, but it can be found in the first version of
the paper on the arXiv, namely {\tt arXiv:1707.08603v1}.)
\end{remark}

\section{Conclusions}

What conclusions can we draw from these results?
First of all, they demonstrate the primordial role played by the orthogonality property
Theorem~\ref{T:orthog}. Indeed,  the equivalence between conditions (a) and (b)
in Theorem~\ref{T:2x2} demonstrates that this property completely
characterizes exceptional $2\times 2$  matrices 
(among matrices $E$ with $\sigma(E)\subset\DD$ and $\|E\|>1$).
For matrices of higher dimension, this is no longer the case.
Nonetheless, the partial characterization  in
Theorem~\ref{T:NxNnec} seems to indicate that, apart from the orthogonality property,
exceptional matrices satisfy no other equality-type constraints.
There remain the inequality-type constraints, such as those in Theorem~\ref{T:positive}.
Perhaps these might eventually lead to a complete characterization of exceptional
matrices in all dimensions.

\section*{Funding}
Ransford's work  was supported by an NSERC Discovery Grant  and by the Canada Research Chairs Program. 
Walsh's work was  supported by an NSERC Undergraduate Student Research Award and an FRQNT Supplement.

\bibliographystyle{tfnlm}
\bibliography{biblist}

\end{document}